\documentclass[a4paper,10pt]{article}
\usepackage{amsmath,amsthm,amssymb,amscd,layout}   
\usepackage{amssymb,bm,color,enumerate} 
\usepackage[dvips]{graphicx}
\newtheorem{theorem}{Theorem}[section]
\newtheorem{lemma}[theorem]{Lemma}
\newtheorem{proposition}[theorem]{Proposition}

\theoremstyle{definition}
\newtheorem{definition}[theorem]{Definition}
\newtheorem{example}[theorem]{Example}
\theoremstyle{remark}
\newtheorem{remark}[theorem]{Remark}
\theoremstyle{conjecture}

\numberwithin{equation}{section}
\def\R{{\mathbb R}}
\def\N{{\mathbb N}}
\newcommand{\qbinom}[3]
{\left [ \begin{array}{c} #1 \\ #2\end{array}\right]_{#3}}

\title{\bf Radial Bargmann representation\\ for 
the Fock space of type B}
\author{Nobuhiro ASAI\footnote{
Department of Mathematics,
Aichi University of Education, 
Hirosawa 1, Igaya, 
Kariya 448-8542, Japan.
Email: nasai[at]auecc.aichi-edu.ac.jp
},
Marek BO\.ZEJKO\footnote{
Institute of Mathematics, 
University of Wroc\l{}aw, 
Pl. Grunwaldzki 2/4, 50-384 Wroc\l{}aw, Poland.
{Email: marek.bozejko[at]math.uni.wroc.pl}
}, 
and 
Takahiro HASEBE\footnote{
Department of Mathematics, 
Hokkaido University, 
Kita 10, Nishi 8, Kita-ku, 
Sapporo 060-0810, Japan. 
{Email: thasebe[at]math.sci.hokudai.ac.jp}
}
}
\date{}
\begin{document}

\maketitle
\begin{abstract}
Let $\nu_{\alpha,q}$ be the probability and orthogonality measure
for the $q$-Meixner-Pollaczek orthogonal polynomials,
which has appeared in \cite{BEH15} as 
the distribution of the $(\alpha,q)$-Gaussian
process (the Gaussian process of type B) 
over the $(\alpha,q)$-Fock space (the Fock space of type B).
The main purpose of this paper is to 
find the radial Bargmann representation of $\nu_{\alpha,q}$.
Our main results cover not only the representation
of $q$-Gaussian distribution by \cite{LM95},
but also of $q^2$-Gaussian and 
symmetric free Meixner distributions on $\mathbb R$.
In addition, non-trivial commutation relations 
satisfied by $(\alpha,q)$-operators are  
presented.

\medskip
\noindent
\textbf{Keywords}: 
Radial Bargmann representation, deformation, Fock spaces,
$q$-orthogonal polynomials.

\smallskip
\noindent
{\bf 2010 Mathematics Subject Classification}: 
33D45, 46L53, 60E99.

\end{abstract}
\section{Introduction}
Bo\.zejko-Ejsmont-Hasebe \cite{BEH15} 
considered a deformation of the (algebraic) full Fock space 
with two parameters $\alpha$ and $q$, namely,
the $(\alpha,q)$-Fock space (or the Fock space of type B)
${\mathcal F}_{\alpha,q}(H)$
over a complex Hilbert space $H$.
The deformation with $\alpha=0$ is equivalent to the
$q$-deformation by Bo\.zejko-Speicher \cite{BS91}
and Bo\.zejko-K\"ummerer-Speicher \cite{BKS97},
and the corresponding $q$-Bargmann-Fock space  
has been constructed by van Leeuwen-Maassen \cite{LM95}.

For the construction of ${\mathcal F}_{\alpha,q}(H)$,
their starting point 
is to replace the Coxeter group of type A,
that is, symmetric group $S_n$ for the $q$-Fock space
by the Coxeter group of type B, 
$\Sigma_n:={\mathbb Z}^n_2\rtimes S_n$
in \eqref{eq:coxeterB} of the Appendix \ref{sec:appendix}.
This replacement provides us 
a more general symmetrization operator on 
$H^{\otimes n}$ than that of \cite{BS91}
to define the $(\alpha,q)$-inner product 
$\langle \cdot,\cdot\rangle_{\alpha,q}$
in \eqref{eq:aq-inner product}.  One can 
define annihilation $B^-_{\alpha,q}(f)$ and 
creation $B^+_{\alpha,q}(f)$ operators acting on 
${\mathcal F}_{\alpha,q}(H)$,
and the $(\alpha,q)$-Gaussian process 
(the Gaussian process of type B) 
$G_{\alpha,q}(f)$ for $f\in H$
as the sum of them,
$G_{\alpha,q}(f):=B^-_{\alpha,q}(f)+B^+_{\alpha,q}(f)$.
It is one of their main interests to find a 
probability distribution $\mu_{\alpha,q,f}$
on $\mathbb R$ of $G_{\alpha,q}(f)$,
$\|f\|_H=1$, with respect to the vacuum state
$\langle \Omega,\cdot~\Omega\rangle_{\alpha,q}$.
${\mathcal F}_{\alpha,q}(H)$ 
equipped with $\langle \cdot,\cdot\rangle_{\alpha,q}$,
$B^-_{\alpha,q}(f)$, and $B^+_{\alpha,q}(f)$
is a typical example of 
interacting Fock spaces in the sense of 
Accardi-Bo\.zejko \cite{AB98}.
It suggests that the theory of orthogonal 
polynomials plays intrinsic roles 
in all previous works mentioned above.
In fact, the measure $\mu_{\alpha,q,f}$ 
given in \cite[Theorem 3.3]{BEH15}
is derived essentially from the orthogonality measure 
$\nu_{\alpha,q}$ associated 
with the $q$-Meixner-Pollaczek orthogonal polynomials
$\{P_n^{(\alpha,q)}(x)\}$ for $\alpha,q\in (-1,1)$ 
given by the recurrence relation,
\begin{equation*}
\begin{cases}
	P_0^{(\alpha,q)}(x)=1, \ P_1^{(\alpha,q)}(x)=x,\\
	xP_n^{(\alpha,q)}(x)=P_{n+1}^{(\alpha,q)}(x)
	+ (1+\alpha q^{n-1})[n]_q
	P_{n-1}^{(\alpha,q)}(x),
	\quad n\geq 1
\end{cases}
\end{equation*}
where 
$[n]_q=1+q+\cdots +q^{n-1}$ is the $q$ number.
However, the Bargmann representation (measure on $\mathbb C$)
of $\nu_{\alpha,q}$ 
has not been obtained yet except the case of $\alpha=0$
for $0\leq q <1$ \cite{LM95}, 
for $q=1$ \cite{Barg61}\cite{AKK03}, for $q=0$ \cite{Bi97},
and $t$-deformed cases of these \cite{AKW16}\cite{KW14},
and for $q>1$ \cite{Kr98}.

Therefore, the main purpose of this paper is to 
find the radial Bargmann representation  
of the probability measure  
$\nu_{\alpha,q}$ on $\mathbb R$.
Our results cover 
the radial Bargmann representations of $q$-Gaussian, 
symmetric free Meixner (Kesten) and $q^2$-Gaussian
distributions on $\mathbb R$.

The organization of this paper will be as follows.
In Section \ref{sec:key}, we shall explain 
how the $(\alpha,q)$-Fock space is related
to the notion of one-mode interacting Fock spaces and 
Bargmann representation.
In Section \ref{sec:aq-Bargmann rep},
the radial Bargmann representation 
of $\nu_{\alpha,q}$ is constructed explicitly
in Theorem \ref{thm main}.
In Section \ref{sec:aq-commutation},
commutation relations 
satisfied by one-mode $(\alpha,q)$-annihilation and 
creation operators 
will be treated.
In the Appendix, we shall give a minimum reference
on the Coxeter group of type B extracted from \cite{BEH15}.

\section{Key ideas and our purpose}\label{sec:key}
Let us point out some of the keys 
to calculate the distribution of 
$G_{\alpha,q}(f)$ in \cite{BEH15}.
It is shown that a linear map,  
$\Phi: \text{Span}\{f^{\otimes n} \ | \ f\in H, \ n\geq 0\}
\to L^2(\mathbb R,\mu_{\alpha,q,f})$
given by 
$\Phi(f^{\otimes n})
=P_n^{(\alpha\langle f,\overline{f}\rangle_H,q)}(x)$,
is an isometry and 
a relation under $\|f\|_H=1$,
\begin{align*}
	G_{\alpha,q}(f)f^{\otimes n}
	& =(B^+_{\alpha,q}(f)+B^-_{\alpha,q}(f))(f^{\otimes n})\\
	&= f^{\otimes (n+1)}
	+(1+\alpha\langle f,\overline{f}\rangle_H)q^{n-1}[n]_q
	f^{\otimes (n-1)},
\end{align*}
is satisfied 
where $\overline{f}$ denotes a self-adjoint involution
of $f\in H$ in \eqref{eq:pi-action}.
This corresponds to the three terms recursion relation 
satisfied by $P_n^{(\alpha\langle f,\overline{f}\rangle_H,q)}(x)$ 
through $\Phi$.
Then, it is proved that   
$\mu_{\alpha,q,f}=\nu_{\alpha\langle f,\overline{f}\rangle_H,q}$
(see $\nu_{\alpha,q}$ in \eqref{eq:qMPmeasure}) in the sense of 
\begin{equation}\label{eq:aq-vacuum}
	\langle
	\Omega, 
	G_{\alpha,q}(f)^n
	\Omega
	\rangle_{\alpha,q}
	=\int x^n\mu_{\alpha,q,f}(dx)
\end{equation}
where $\Omega$ denotes the vacuum vector. 
Therefore, 
in order to get 
the Bargmann representation of $\nu_{\alpha\langle f,\overline{f}\rangle_H,q}$,
it is enough to consider 
that of $\nu_{\alpha,q}$
in the sense of Definition \ref{def:bargmann-def} 
given later.

Since the structure mentioned above can be reduced to 
the one-mode analogue of $(\alpha,q)$-Fock spaces,
let us recall fundamental relationships between 
one-mode interacting Bargmann-Fock spaces and the theory of orthogonal polynomials of one variable.

\begin{definition}\label{def:1-mode}
Let $\{\omega_n\}_{n=0}^\infty$ with $\omega_0:=1$ be an infinite 
sequence of positive real
numbers and $\{\alpha_n \}_{n=0}^\infty$ be of real numbers. 
A one-mode interacting Bargmann-Fock space $\mathcal B$
is defined as 
$\bigoplus_{n=0}^\infty{\mathbb C}\Phi_n$
equipped with $\Phi_n :=z^n/[\omega_n]!, \ [\omega_n]!:=\prod_{k=0}^n\omega_k$, the inner product 
$\langle\Phi_m,\Phi_n\rangle_{\mathcal B}=\delta_{m,n}$
for all $m,n \in \mathbb N\cup \{0\}$,
operators of creation $a^+$, annihilation $a^-$, 
and conservation $a^{\circ}$ defined by
\begin{equation}\label{eq:1-mode-fs}
\begin{cases}
	a^+\Phi_n := \sqrt{\omega_{n+1}}\Phi_{n+1}, & n\geq 0,\\
	a^{-}\Phi_0=0, \ 
	a^{-}\Phi_n := \sqrt{\omega_n}\Phi_{n-1}, &  n\geq 1,\\
	a^{\circ}\Phi_n := \alpha_n\Phi_n, &  n\geq 0.
\end{cases}
\end{equation}
\end{definition}

Let $(\{\omega_n\}_{n=0}^\infty, \{\alpha_n \}_{n=0}^\infty)$
be a pair of sequences in Definition \ref{def:1-mode}
and define a sequence of monic polynomials
$\{P_n(x)\}$ recurrently by 
\begin{equation}
\begin{cases}
	P_0(x) = 1, \ P_1(x)=x-\alpha_0,\\
	xP_n(x)= P_{n+1}(x)+\omega_nP_{n-1} + \alpha_{n}P_n(x),
	\quad n\geq 1. \label{eq:3terms}
\end{cases}
\end{equation} 
Then there exists a probability measure $\mu$ 
on ${\mathbb R}$ with finite moments of all orders such that 
$\{P_n(x)\}$ is the orthogonal polynomials with 
$\langle P_m(x),P_n(x)\rangle_{L^2(\mathbb R,\mu)}
= \delta_{m,n} [\omega_n]!$ for all 
$m,n \in \mathbb N\cup \{0\}$. 
(See \cite{Chi78}\cite{HO07}, for example.)

It is easy to see that a linear map 
\begin{equation*}
	U : {\mathcal B}=\bigoplus_{n=0}^\infty{\mathbb C}\Phi_n
	\to L^2({\mathbb R},\mu)
\end{equation*}
defined by 
$U\left(\sqrt{[\omega_n]!}\Phi_n\right)=P_n(x)$
is an isometry and in addition
$a^+ +a^- +a^{\circ}=U^*XU$ is satisfied 
due to \eqref{eq:1-mode-fs} and \eqref{eq:3terms},
where $X$ represents the multiplication operator 
by $x$ in $L^2(\mathbb R,\mu)$.
This intertwining relation provides a notion of 
the quantum decomposition of a 
classical random variable $X$ and  
\begin{equation}\label{eq:1-mode vacuum}
	\langle
	\Phi_0, (a^+ +a^- +a^{\circ})^n\Phi_0
	\rangle_{\mathcal B}
	=\int x^n\mu(dx).
\end{equation}
Therefore,  
if $\omega_n=(1+\alpha q^{n-1})[n]_q, \ \alpha_n=0$,
the equality in \eqref{eq:1-mode vacuum}
is a one-mode analogue of \eqref{eq:aq-vacuum}.

Now it is interesting to consider the following
moment problem to realize the inner product by 
the integral:

\medskip
\noindent
{\bf Problem 1.}
For a given $\{\omega_n\}$ of $\mu$, 
find a probability measure $\gamma_{\mu}$
satisfying the equality,
\begin{equation}\label{eq:bargmann-def}
	\int_{\mathbb C}\overline{z}^mz^n\gamma_{\mu}(d^2z)
	=\delta_{m,n}[\omega_n]!
\end{equation}
for all $m,n \in \mathbb N\cup \{0\}$.

\begin{definition}\label{def:bargmann-def}
A measure $\gamma_{\mu}$ satisfying 
the equality \eqref{eq:bargmann-def}
is called a Bargmann representation (measure on $\mathbb C$)
of a measure $\mu$ on $\mathbb R$.
\end{definition}

It was proved 
in \cite{Sz07} (see also \cite{AKW16}\cite{KW14}) 
that if a measure 
$\mu$ admits any Bargmann representation, then it also admits a 
radial (rotation invariant) Bargmann representation
\begin{equation*}
	\gamma_\mu(d^2z)=\frac{1}{2\pi}
	\lambda_{[0,2\pi)}(d\theta)\rho_\mu(dr),
	\ z=re^{i\theta}, \ r\geq0, \ \theta\in[0,2\pi), 
\end{equation*}
where $\lambda_{[0,2\pi)}$ is the Lebesgue measure on $[0,2\pi)$. 
It says that the angular part takes care of orthogonality of
\eqref{eq:bargmann-def}.  Therefore, 
Problem 1 can be transformed into
the following Problem 2:

\medskip
\noindent
{\bf Problem 2.}
Find a positive radial measure $\rho_\mu$ satisfying 
\begin{equation*}
	\int_0^\infty r^{2n}\rho_\mu(dr)
	=[\omega_n]!
\end{equation*}
for all $m,n \in \mathbb N\cup \{0\}$. 

\medskip
\noindent
{\bf Main Purpose:}
We shall consider Problem 2
associated with 
$\omega_n=(1+\alpha q^{n-1})[n]_q, \ \alpha_n=0$
of $\nu_{\alpha.q}$ in Section \ref{sec:aq-Bargmann rep}.
Furthermore, commutation relations 
satisfied by $a^+,a^-$ acting on $\mathcal B$
associated with $\omega_n=(1+\alpha q^{n-1})[n]_q$
will be presented in Section \ref{sec:aq-commutation}.

\begin{remark}
\begin{enumerate}[\rm(1)]
\item
One can notice that $\gamma_{\mu}$ and $\rho_\mu$ are
determined only 
by $[\omega_n]!$.  Therefore, 
it is enough in general 
for the Bargmann representation in the above sense 
to consider the symmetric measure $\mu$ with $\alpha_n=0$
for all $n$, which implies that $a^\circ$ is a zero operator.
\item
If $\mu$ is symmetric, then $\alpha_n=0$
for all $n$ is implied.  The converse statement
is true if $\mu$ is determined by its moments. 
\end{enumerate}
\end{remark}
\section{$(\alpha,q)$-Bargmann representation}
\label{sec:aq-Bargmann rep}
\subsection{$q$-Meixner-Pollaczek polynomials}
Let us recall standard notations from $q$-calculus,
which can be found in \cite{GR04}\cite{KLS10} for example.
The $q$-shifted factorials are defined by 
\begin{equation*}
	(a;q)_0 :=1, \quad
	(a;q)_k:=\prod_{\ell=1}^{k}(1-aq^{\ell -1}),
	\ k=1,2,\dots,\infty,
\end{equation*}
and the product of $q$-shifted factorials is defined by
\begin{equation*}	
	(a_1,a_2;q)_k := (a_1;q)_k(a_2;q)_k,
	\quad
	\ k=1,2,\dots,\infty.
\end{equation*}

\begin{remark}\label{rem:Pochhammer}
The $q$-shifted factorials are a natural extension
of the Pochhammer symbol $(\cdot)_n$ because one can see
that $\lim_{q\to 1}[k]_q=k$ implies
\begin{equation}\label{eq:Pochhammer}
	\lim_{q\to 1}\frac{(q^k;q)_n}{(1-q)^n}=(k)_n,
\end{equation}
where $(k)_0:=1, \ (k)_n:=k(k+1)\cdots(k+n-1), \ n\geq 1$. 
\end{remark}

As we have mentioned,  
$\{P_n^{(\alpha,q)}(x)\}$ for $\alpha,q\in (-1,1)$ 
is the $q$-Meixner-Pollaczek polynomials
satisfying the recurrence relation,
\begin{equation}\label{eq:P_n^q}
\begin{cases}
	P_0^{(\alpha,q)}(x)=1, \ P_1^{(\alpha,q)}(x)=x,\\
	xP_n^{(\alpha,q)}(x)=P_{n+1}^{(\alpha,q)}(x)
	+ (1+\alpha q^{n-1})[n]_q
	P_{n-1}^{(\alpha,q)}(x),
	\quad n\geq 1.
\end{cases}
\end{equation}
It is known in \cite[14.9.2]{KLS10} 
and \cite[page 1781]{BEH15}
that 
the orthogonality measure $\nu_{\alpha.q}$ for such polynomials 
has the density of the form,
\begin{equation}\label{eq:qMPmeasure}
	\frac{(q,\gamma^2;q)_\infty}{2\pi}
	\sqrt{\frac{1-q}{4-(1-q)x^2}}
	\left(
	\frac{g(x,1;q)g(x,-1;q)g(x,\sqrt{q};q)g(x,-\sqrt{q};q)}
	{g(x,i\gamma;q)g(x,-i\gamma;q)}
	\right),
\end{equation}
supported on the interval $(-2/\sqrt{1-q},2/\sqrt{1-q})$
where 
\begin{equation*}
	g(x,b;q)=\prod_{k=0}^\infty
	(1-4bx(1-q)^{-1/2}q^k+b^2q^{2k}),
\end{equation*}
and 
\begin{equation*}
	\gamma=
\begin{cases}
	\sqrt{-\alpha}, & \alpha< 0,\\
	i\sqrt{\alpha}, & \alpha\geq 0.
\end{cases}
\end{equation*}

\begin{example}
\begin{enumerate}[(1)]
\item
If $\alpha=0$, then $q$-Meixner-Pollaczek
polynomials get back to the $q$-Hermite polynomials $H_n^{(q)}(x)$ 
whose orthogonality measure is the standard $q$-Gaussian measure
on $(-2/\sqrt{1-q},2/\sqrt{1-q})$
given by
\begin{align*}
	\nu_q(dx)
	& := \frac{\sqrt{1-q}}{\pi}
	\sin\theta 
	\prod_{n=1}^{\infty}(1-q^n)|1-q^ne^{2i\theta}|^2 dx,
\end{align*}
where $x\sqrt{1-q}=2\cos\theta, \ \theta\in [0,\pi]$.
Furthermore, 
one can get the standard Gaussian law as $q\to 1$, 
 the Bernoulli law as $q\to -1$,
and the standard Wigner's semi-circle law if $q=0$. 
See \cite{BKS97}\cite{BS91}.

\item 
The measure $\nu_{\alpha,0}$ is the 
symmetric free Meixner law
\cite{An03}\cite{BB06}\cite{SY01}. 

\item The measure $\nu_{q,q}$ is the $q^2$-Gaussian 
law scaled by $\sqrt{1+q}$.

\item 
If $\alpha=-q^{2\beta}$
as suggested in Remark \ref{rem:Pochhammer}, then
the measure $\nu_{-q^{2\beta},q}$
under a certain scaling  
converges to the classical symmetric Meixner law
as $q\uparrow 1$, 
\begin{equation}\label{eq:symmetric-meixner}
	\frac{2^{2\beta}}{2\pi\varGamma(2\beta)}
	|\varGamma(\beta+ix)|^2 dx, 
	\quad x\in {\mathbb R}.
\end{equation}
See also \cite[14.9.15]{KLS10}.
\end{enumerate}
\end{example}

%
\subsection{Problem}
For $\alpha,q \in(-1,1)$, we would like to know when 
there exists a radial measure $\rho_{\nu_{\alpha,q}}$ 
satisfying   
\begin{equation}\label{prob:aq-radial}
	\int_0^\infty r^{2 k} \,\rho_{\nu_{\alpha,q}} (d r) 
	=(-\alpha;q)_k[k]_q!,
	\qquad k \in {\mathbb N} \cup\{0\}.
\end{equation}
Here $[k]_q!$ denotes the $q$-factorials defined by 
\begin{equation*}
	[0]_q!:=1, 
	\quad [k]_q!:=\prod_{\ell =1}^k[\ell]_q=\frac{(q;q)_k}{(1-q)^k},
	\ k\geq 1.
\end{equation*}
It is easy to get the inequality for $\alpha, q\in (-1,1)$, 
\begin{equation}\label{eq:aq-ineq}
	\left|(-\alpha;q)_k[k]_q!\right|
	\leq \left(\frac{4}{1-|q|}\right)^k, 
	\ k\in {\mathbb N}\cup \{0\}.
\end{equation}	
Due to Carleman criterion for the moment problem, 
this inequality implies that 
a radial measure $\rho_{\nu_{\alpha,q}}$ is determined 
uniquely by the sequence 
$\left\{(-\alpha;q)_k[k]_q!\right\}$.

We shall 
follow the procedure below to construct $\rho_{\nu_{\alpha,q}}$  
in \eqref{prob:aq-radial}. 
\begin{enumerate}[\rm(1)]
\item
Recall the radial part of the $q$-Gaussian measure
on $\mathbb C$ ($q$-Bargmann measure), $\rho_{\nu_q} = \rho_{\nu_{0,q}}$, obtained in \cite{LM95},
\begin{equation}\label{eq:radial-qgauss}
	\int_0^\infty r^{2 k} \, \rho_{\nu_q}(dr) 
	= [k]_q!.
\end{equation}
\item
Find a radial (possibly signed) measure $\rho_{\alpha,q}$
having the moment $(-\alpha;q)_k$.
\item
Compute the multiplicative (Mellin) convolution $\rho_{\nu_q}\circledast \rho_{\alpha,q}$ to 
get $\rho_{\nu_{\alpha,q}}$.
\end{enumerate}

\begin{remark}
It is known \cite{LM95} that a radial measure $\rho_{\nu_q}$ 
in \eqref{eq:radial-qgauss}
does not exist for $q<0$.
However, one can see that the positivity assumption 
on $q$ can be relaxed for $\rho_{\nu_{\alpha,q}}$ if $\alpha=q$.
It will be discussed right after the proof of 
Proposition \ref{bargmann} and in Proposition \ref{thm a=q}.
\end{remark}

\subsection{Construction of $(\alpha,q)$-radial measures}
\begin{lemma}\label{lem bargmann} 
Suppose that $\alpha\in(-1,1)$ and $q\in [0,1)$. Let 
\begin{equation*}
	\rho_{\alpha,q}
	:= (-\alpha;q)_\infty
	\sum_{n=0}^\infty \frac{(-\alpha)^n}{(q;q)_n} 
	\delta_{q^{n/2}}, 
\end{equation*}
which is a signed measure. Then we have 
\begin{equation*}
	\int_0^\infty r^{2 k} \, \rho_{\alpha,q}(dr) 
	= (-\alpha;q)_k,\qquad k \in {\mathbb N}\cup\{0\}. 
\end{equation*}
In particular, if taking $\alpha=-q$, then one can see 
$\rho_{\nu_q}=D_{(1-q)^{-1/2}}(\rho_{-q,q})$, namely, 
\begin{equation*}
	\int_0^\infty r^{2 k} \, D_{(1-q)^{-1/2}}(\rho_{-q,q}) (dr) 
	= \frac{(q;q)_k}{(1-q)^k} = [k]_q!,
\end{equation*}
where $D_t(\lambda)$ is the push-forward of $\lambda$ 
by the map $x\mapsto t x$ for a measure $\lambda$ on $\mathbb R$.
\end{lemma}

\begin{proof} 
Firstly, we have  
\begin{equation*}\label{eq A}
\int_0^\infty r^{2 k} \, \rho_{\alpha,q}(dr) = (-\alpha;q)_\infty\sum_{n=0}^\infty \frac{(-\alpha q^k)^n}{(q;q)_n}. 
\end{equation*}
Since  Euler's formula (see \cite[1.3.15]{GR04}),
\begin{equation}\label{eq B}
	\frac{1}{(a;q)_\infty} =\sum_{n=0}^\infty \frac{a^n}{(q;q)_n},
\end{equation}
is known, we replace $a$ by $-\alpha q^k$ 
in \eqref{eq B} to obtain 
\begin{align*} 
	\int_0^\infty r^{2 k} \, \rho_{\alpha,q}(dr) 
	& =\frac{(-\alpha;q)_\infty}{(-\alpha q^k;q)_\infty}\\
	& = (-\alpha;q)_k. 
\end{align*}
The proof is complete.
\end{proof}

\begin{remark}
\begin{enumerate}[\rm(1)]
\item 
The last equality in proof is due to the formula  
\begin{equation*}
	(a;q)_k=\frac{(a;q)_\infty}{(aq^k;q)_\infty}.
\end{equation*}
See \cite[1.2.30]{GR04}, for example. 
\item
Euler's formula is considered as 
the $q$-analogue of exponential function $e^a$ due to 
\begin{equation*}
	\lim_{q\to 1}\frac{1}{((1-q)a;q)_n}=e^a.
\end{equation*}
\end{enumerate}
\end{remark}

Let 
\begin{equation*}
	\qbinom{n}{\ell}{q} 
	:=\frac{[n]_q!}{[\ell]_q![n-\ell]_q!}
	=\frac{(q;q)_n}{(q;q)_\ell(q;q)_{n-\ell}} 
\end{equation*}
be the $q$-binomial coefficients and 
$h_n(z\mid q)$ be the Rogers-Szeg\"o polynomials 
\cite{GR04}\cite{S05}
defined by 
\begin{equation*}
	h_n(z\mid q)
	= \sum_{\ell=0}^n \qbinom{n}{\ell}{q} z^\ell.
\end{equation*}

\begin{proposition}\label{bargmann} 
Suppose that $\alpha \in(-1,1)$ and $q\in[0,1)$. Let 
\begin{equation}\label{bargmann2}
	\rho_{\nu_{\alpha,q}}
	:=
\begin{cases} \displaystyle
(-\alpha,q;q)_\infty 
\sum_{n=0}^\infty 	
\frac{q^n}{(q;q)_n} h_n(-\alpha q^{-1}\mid q) 
\delta_{(1-q)^{-1/2}q^{n/2}}, & q>0, \\
 -\alpha \delta_0+(1+\alpha)\delta_1, & q=0,  
\end{cases}
\end{equation}
which is a signed measure in general. Then we have
\begin{equation}\label{eq moments}
\int_0^\infty r^{2 k} \, \rho_{\nu_{\alpha,q}}(d r) 
= \frac{(-\alpha,q;q)_k}{(1-q)^k} 
=(-\alpha;q)_k[k]_q!, \qquad k \in \N\cup\{0\}. 
\end{equation}
\end{proposition}
\begin{proof} First of all, it is easy to show \eqref{eq moments}
for the case $q=0$.  Therefore, let us  assume $q>0$. 

One can compute   
the multiplicative (Mellin) convolution $\circledast$ 
to get $\rho_{\nu_{\alpha,q}}$ as follows:
\begin{align*}
	\rho_{\nu_{\alpha,q}}
	& = \rho_{\alpha,q} \circledast D_{(1-q)^{-1/2}}(\rho_{-q,q})\\
	& = (-\alpha,q;q)_\infty \sum_{n=0}^\infty 
	\left
	(\sum_{\ell=0}^n 
	\frac{(-\alpha)^{\ell} q^{n-\ell}}{(q;q)_{\ell} (q;q)_{n-\ell}
	}\right)
	\delta_{(1-q)^{-1/2}q^{n/2}}\\
	& =(-\alpha,q;q)_\infty 
	\sum_{n=0}^\infty 	
	\frac{q^n}{(q;q)_n} 
	h_n(-\alpha q^{-1}\mid q)
	\delta_{(1-q)^{-1/2}q^{n/2}}.
\end{align*}
On the other hand, by Lemma \ref{lem bargmann}, we have
\begin{equation*}
\int_0^\infty r^{2 k} \, D_{(1-q)^{-1/2}}(\rho_{-q,q}) (d r) =  \frac{(q;q)_k}{(1-q)^k} = [k]_q!. 
\end{equation*}
Therefore, one can get
\begin{equation*}
\begin{split}
\int_0^\infty r^{2 k} \,\rho_{\nu_{\alpha,q}} (d r) 
&=\int_0^\infty r^{2 k} \, \rho_{\alpha,q}(d r) \int_0^\infty r^{2 k} \, D_{(1-q)^{-1/2}}(\rho_{-q,q}) (d r) \\
&=  (-\alpha;q)_k [k]_q!,\qquad k \in\N \cup\{0\}. 
\end{split}
\end{equation*}
\end{proof}

In Proposition \ref{bargmann},
we have obtained $\rho_{\nu_{\alpha,q}}$
for $\alpha \in(-1,1)$ and $q\in(0,1)$.
Due to the term
$$
	\delta_{(1-q)^{-1/2}q^{n/2}}
	\ \  \text{in} 
	\ \  \rho_{\nu_{\alpha,q}},
$$ 
it seems impossible for $q\in(-1,0)$ to define $\rho_{\nu_{\alpha,q}}$. 
However, if $-1<\alpha=q<0$ then  $\nu_{q,q}$ coincides with a scaled $q^2$-Gaussian measure, and hence the Bargmann measure exists.


\begin{proposition}\label{thm a=q}
Suppose $-1< \alpha=q<0$. We define 
\begin{equation}\label{bargmann4}
\begin{split}
\rho_{\nu_{q,q}} 
&:=D_{(1+q)^{1/2}}(\rho_{\nu_{q^2}})  \\
&=  (q^2;q^2)_\infty 
	\sum_{n=0}^\infty 
	\frac{q^{2n}}{(q^2;q^2)_n} 
	\delta_{(1-q)^{-1/2}(-q)^{n}}.
\end{split}	
\end{equation}
Then one can see 
\begin{equation*}
	\int_0^\infty r^{2 k} \, \rho_{\nu_{q,q}}(dr)  = (1+q)^k [k]_{q^2}! 
	 = (-q;q)_k [k]_q!. 
\end{equation*}
\end{proposition}
\begin{proof} One can see by direct computations  
\begin{align*}
	(-q;q)_k [k]_q! 
	&= \left\{\prod_{\ell=1}^k(1-(-q)q^{\ell-1})\right\}
	\left\{\prod_{\ell=1}^k\frac{1-q^\ell}{1-q}\right\}\\
	& = (1+q)^k\prod_{\ell=1}^k\frac{1-q^{2\ell}}{1-q^2}\\
	& = (1+q)^k [k]_{q^2}!.
\end{align*}
Thus $\rho_{\nu_{q,q}}$ can be defined as 
the radial measure for $q^2$-Gaussian measure on $\mathbb C$ 
scaled by $(1+q)^{1/2}$, namely, $\rho_{\nu_{q,q}} = D_{(1+q)^{1/2}}(\rho_{\nu_{q^2}})$. 
\end{proof}

\begin{remark} 
If we use the fact that $h_n(-1 \mid q)=0$ 
for odd $n\geq1$ (see proof of Lemma \ref{lem Bargmann} below), 
we can extend the definition \eqref{bargmann2} 
to the case $-1<\alpha=q<0$. 
This will give an alternative way 
to define $\rho_{\nu_{q,q}}$ for $-1<q<0$,
but both definitions give the same measure. 
\end{remark}

We need some properties of the Rogers-Szeg\"o
 polynomials to know when the measure $\rho_{\nu_{\alpha,q}}$ becomes positive. 

\begin{lemma}[\cite{MGH90}]\label{lem Bargmann} 
Suppose that $q\in(-1,1)$. 
\begin{enumerate}[\rm(1)]
\item\label{lem Bargmann1} If $n\geq0$ is odd,
then $h_n(x\mid q)\geq0$ if and only if $x\geq-1$.  
Moreover, the point $x=-1$ is the unique zero of $h_n(x\mid q)$ on $\R$. 
\item\label{lem Bargmann2} 
If $n\geq0$ is even, then $h_n(x\mid q)>0$ for all $x \in\R$. 
\end{enumerate} 
\end{lemma}
\begin{proof}
It is known that all the zeros of $h_n(z\mid q)$ lie on the unit circle $|z|=1$.  See \cite{MGH90} or \cite[Theorem 1.6.11]{S05}. 
Note that the result was obtained for $q\in[0,1)$,
but the proof can be extended to $q\in(-1,1)$ without any modifications. 

By definition, one can see
\begin{equation*}
\qbinom{n}{\ell}{q} = \frac{(1-q^{n-\ell+1})(1-q^{n-\ell+2})\cdots (1-q^n)}{(1-q)(1-q^2)\cdots (1-q^\ell)}>0, 
\end{equation*}
and hence $h_n(1\mid q)>0$ for all $n\geq0$. 
Thus, $h_n(x\mid q) \neq0$ for $x\in \R\setminus\{-1\}$.   
It then suffices to show that $h_n(-1\mid q)>0$ for all even $n\geq0$ and $h_n(-1\mid q)=0$ for all odd $n \geq1$. 
The recurrence relation 
for the Rogers-Szeg\"{o} polynomials is known to be 
\begin{equation}\label{Rec RS}
h_{n+1}(z\mid q) = (z+1) h_n(z\mid q) -(1-q^n)z h_{n-1}(z\mid q),
\qquad n\geq1.
\end{equation}
See \cite[1.6.76]{S05} (note that formula (1.6.76) has an error of a sign). 
It is easy to see that $h_0(-1\mid q)=1>0, h_1(-1\mid q)=0$, 
so by induction and \eqref{Rec RS} one can show  $h_n(-1\mid q)>0$ for all even $n\geq0$ and $h_n(-1 \mid q)=0$ for all odd $n \geq1$.
\end{proof}

We need the following lemma in proof of Theorem \ref{thm main} 
for the non-existence part of a radial Bargmann measure.

\begin{lemma} \label{lem unique}
Let $\mu$ be a signed measure on $\R$ with compact support and let $\nu$ be a nonnegative measure on $\R$. If $\mu$ and $\nu$ have the same finite moments of all orders, then $\mu=\nu$. 
\end{lemma}
\begin{proof}
We  denote by $m_n$ the moments of $\mu$ (and $\nu$ by assumption). Since $\mu$ is compactly supported, say on $[-R,R]$,  
\begin{equation*}
|m_n| = \left|\int_{[-R,R]} x^n\,\mu(d x) \right| \leq \| \mu \| R^n,\qquad n\in\N\cup\{0\}, 
\end{equation*}
where $\|\mu\|$ denotes the total variation of $\mu$. Therefore, $\nu$ is also supported on $[-R,R]$. By Weierstrass' approximation, we have 
\begin{equation}\label{eq Hahn}
\int_{[-R,R]} f(x)\,\mu(d x) =\int_{[-R,R]} f(x)\,\nu(d x)
\end{equation}
for all $f\in C([-R,R])$. This implies that $\mu=\nu$ 
 since, if we use the Hahn decomposition $\mu=\mu_+ -\mu_-$, then \eqref{eq Hahn} implies 
\begin{equation*}
\int_{[-R,R]} f(x)\,\mu_+(d x) =\int_{[-R,R]} f(x) \,(\nu+\mu_-)(d x), 
\end{equation*} 
and hence $\mu_+ = \nu+\mu_-$ as nonnegative finite measures. 
\end{proof}

In summary, the complete answer to the unique existence 
of a radial Bargmann representation of $\nu_{\alpha,q}$ is
stated as follows:  
\begin{theorem}\label{thm main}
Suppose that $\alpha, q\in(-1,1)$. 
The probability measure $\nu_{\alpha,q}$ 
has a radial Bargmann representation if and only if either
 (i) $q \geq0$ and $\alpha \leq q$ or (ii) $\alpha=q\ne 0$. 

In fact, the radial measure
is given uniquely by 
\begin{equation*}
	\rho_{\nu_{\alpha,q}}= 
\begin{cases}
	-\alpha \delta_0+(1+\alpha)\delta_1 & (\alpha\leq q=0), \\
	\displaystyle
	(-\alpha,q;q)_\infty 
	\sum_{n=0}^\infty 	
	\frac{q^n}{(q;q)_n} 
	h_n(-\alpha q^{-1}\mid q) 
	\delta_{(1-q)^{-1/2}q^{n/2}} & (q>0, \ \alpha< q),\\
	\displaystyle
	(q^2;q^2)_\infty 
	\sum_{n=0}^\infty 
	\frac{q^{2n}}{(q^2;q^2)_n} 
	\delta_{(1-q)^{-1/2}
	|q|^n}
	& (\alpha=q\ne 0).
\end{cases}
\end{equation*}
\end{theorem}

\begin{proof}
{\bf 1.\ Existence and uniqueness.} 
If $q\in [0,1)$ and $\alpha \leq q$, then 
by Proposition \ref{bargmann} and Lemma \ref{lem Bargmann}, 
the signed measure $\rho_{\nu_{\alpha,q}}$ is 
in fact a nonnegative measure 
and becomes the radial part of a Bargmann measure. 
The case $\alpha=q<0$ was discussed in Proposition \ref{thm a=q}.
Due to Carleman criterion for the moment problem,
the inequality given in \eqref{eq:aq-ineq}
guarantees the uniqueness of $\rho_{\nu_{\alpha,q}}$
for these cases.

\noindent
{\bf 2.\ Non-existence.} 
(1) If $q\in (0,1)$ and $\alpha > q$, then 
$\rho_{\nu_{\alpha,q}}$ is not a nonnegative measure and is really a signed measure since $h_n(-\alpha/q\mid q) <0$ for odd integers $n\geq0$ and $q>0$ from Lemma \ref{lem Bargmann}. By Lemma \ref{lem unique}, if a radial Bargmann measure exists, then it must be equal to the signed measure $\nu_{\alpha,q}$.
This is a contradiction. 
Thus, a radial Bargmann measure does not exist. 

\noindent
(2) If $q=0$ and $\alpha>q=0$ then by \eqref{bargmann2} $\nu_{\alpha,0}$ is really a signed measure, and hence by the same argument as above, a radial Bargmann measure does not exist. 

\noindent
(3) Let 
$$
\beta_k(\alpha,q):= (-\alpha;q)_k [k]_q!, \qquad k \geq0, \alpha,q \in(-1,1). 
$$
Given $q<0$ and $\alpha \neq q$, suppose that 
there exists a radial part of a Bargmann measure, $\rho$. Let $\rho^2$ be the push-forward of $\rho$ 
by the map $x\mapsto x^2$. Then,   
\begin{equation}\label{eq087}
\beta_k(\alpha,q)=\int_0^\infty x^{k} \,\rho^2(d x) = \int_0^\infty x^{2 k} \,\rho(d x).   
\end{equation}
By the way, by Proposition \ref{bargmann} it holds that $\beta_k(\alpha,q')  = \int_0^\infty x^{2 k} \,\rho_{\nu_{\alpha,q'}}(d x)$ for any $q'>0$, that is, 
\begin{equation}\label{eq088}
\beta_k(\alpha,q')  = (-\alpha,q';q')_\infty \sum_{n=0}^\infty \frac{(q')^n}{(q';q')_n} h_n(-\alpha (q')^{-1}\mid q') \frac{(q')^{k n}}{(1-q')^k},\qquad q'>0,  
\end{equation}
 which is true even for $q'=q$ by analytic continuation. 

Now let us consider the signed measure  
\begin{equation*}
\mu:= (-\alpha,q;q)_\infty 
\sum_{n=0}^\infty \frac{q^n}{(q;q)_n} 
h_n(-\alpha q^{-1}\mid q) \delta_{(1-q)^{-1}q^n}, 
\qquad\alpha \neq q <0, 
\end{equation*}
supported on the points 
$\frac{q^n}{1-q}$ for $n=0,1,2,3,\dots$. 
Then by \eqref{eq088} for $q'=q$ and by \eqref{eq087}, 
\begin{equation*}
	\int_{\R} x^k \,\mu(d x) 
	= \beta_k (\alpha,q)
	= \int_0^\infty x^{k} \,\rho^2(d x),\qquad k\in\N\cup\{0\}. 
\end{equation*}
By Lemma \ref{lem unique}, the signed measure $\mu$ and the probability measure $\rho^2$ should be equal. However, the support of $\mu$ is not contained in $[0,\infty)$, 
and hence $\mu$ cannot be equal to $\rho^2$. 
This is a contradiction.  
\end{proof}

\begin{example}\label{ex:(4)}
\begin{enumerate}[(1)]
\item 
The radial measure $\rho_{\nu_{0,q}}$ for $q\in [0,1)$ is 
of the $q$-Bargmann \cite{LM95}.

\item 
The radial measure $\rho_{\nu_{q,q}}$ 
for $q\in(-1,1)$ is 
of the $q^2$-Bargmann.

\item
$\displaystyle\lim_{q\uparrow 1}\rho_{\nu_{\alpha,q}}$ 
is of the classical Bargmann \cite{Barg61}\cite{AKK03}.

\item 
Consider $\alpha=-q^{2\beta}, \beta>0$.  
This choice of $\alpha$ is suggested by \eqref{eq:Pochhammer}
in Remark \ref{eq:Pochhammer}. 
In fact, one can see
$$
	\lim_{q\uparrow 1}\frac{(1-q^{2\beta+n-1})[n]_q}{4(1-q)}
	=\frac{1}{4}(n+2\beta-1)n.
$$
This limit sequence is the Jacobi sequence 
of the symmetric Meixner distribution in 
\eqref{eq:symmetric-meixner}, 
so that 
$\rho_{\nu_{-q^{2\beta},q}}$ under suitable 
scaling converges weakly as $q\uparrow 1$
to the radial measure with the density,
\begin{equation*}
	\frac{2\pi r}{\varGamma(2\beta)}
	\int_0^\infty h(r, t/4)e^{-t}t^{2\beta-1}dt
\end{equation*}
where 
\begin{equation*}
	h(r,t)
	=\frac{1}{\pi t}\exp
	\left(
	-\frac{r^2}{t}
	\right), 
	\ r\in {\mathbb R}, \ t>0.
\end{equation*}
This is an integral representation of the radial density 
for the Bessel kernel measure, 
which can be also represented 
by the modified Bessel function \cite{As05}\cite{As09}.

\item \label{ex:(5)}
$\rho_{\nu_{\alpha,0}}$ for 
$\alpha\in (-1,0]$ is the radial measure for 
the symmetric free Meixner distribution.
See Remark \ref{rem:t-prob} below.
\end{enumerate}
\end{example}

\begin{remark}\label{rem:t-prob}
Let $\mu_t$ be a $t$-deformed probability
measure of a probability measure $\mu$ on $\mathbb R$
defined through the Cauchy transform $G_\mu$ of $\mu$, 
\begin{equation*}
	\frac{1}{G_{\mu_t}(z)}
	:= \frac{t}{G_{\mu}(z)}+(1-t)z,
	\quad t\geq 0,
\end{equation*}
examined by Bo\.zejko-Wysocza\'nski \cite{BW98, BW01}.
Krystek-Wojakowski \cite{KW14}
discussed the radial Bargmann representation of a $t$-deformed probability
measure $\mu_t$, $t$-Bargmann representation for short,  
and obtained necessary and sufficient condition
for the admissibility of the representation.
The $t$-Bargmann representation of 
the Kesten measure $\kappa_t$ has the form, 
\begin{equation*}
	\rho_{\kappa_t}	
	= \left(1-\frac{1}{t}\right)\delta_0
	+\frac{1}{t}\delta_{\sqrt{t}},
	\quad t\geq 1.
\end{equation*}
In \cite{AKW16}, 
the $t$-Bargmann representation of a symmetric 
free Meixner law $\varphi_{s,t}$ with 
two positive parameters $s, t$ is treated and 
is admitted if and only if 
$t\geq 1$.  In fact, one can see 
$\rho_{\varphi_{s,t}}= D_s (\rho_{\kappa_t})$
and hence  
$$
	\rho_{\nu_{(1-t)/t,0}}=\rho_{\varphi_{1/\sqrt{t},t}}
	=D_{1/\sqrt{t}}(\rho_{\kappa_t}),
	\quad t\geq 1.
$$
Therefore, the case \eqref{ex:(5)} in Example \ref{ex:(4)} 
can be viewed as a $t$-Bargmann representation, too. 
\end{remark}
Furthermore, let us state the $t$-deformed version 
of Theorem \ref{thm main} 
for $q\ne 0$ without proof:
\begin{proposition}
The $t$-deformed version of $\rho_{\nu_{\alpha,q}}$
for  $q\ne 0$
is given by 
\begin{equation*}
	\left(
	1-\frac{1}{t}
	\right)\delta_0
	+
	\frac{1}{t}\rho_{\nu_{\alpha,q}}, 
	\quad t\geq 1.
\end{equation*}
\end{proposition}
\begin{remark}
The $t$-Bargmann representation of $\nu_q$ is 
treated in \cite{KW14} for $q=1$
and \cite{AKW16} for $0\leq q<1$.
\end{remark}

Before closing this section, let us give a short remark 
about relations with the free infinite divisibility.
Many of particular examples have so far suggested 
that the free infinite divisibility of a probability measure 
implies the existence of a radial Bargmann representation.
The converse is not true in general because
the Askey-Wimp-Kerov distribution $\mu_{9/10}$ for instance, 
discussed in \cite{BBLS11}, 
is not freely infinitely divisible, 
but it has a Bargmann representation with a gamma 
distribution as its radial measure.
However, not many counterexamples have been found. 

Therefore, we conjecture that the free infinite divisibility of 
our $(\alpha,q)$-Gaussian distribution is equivalent 
to the existence of its radial Bargmann measure:

\medskip
\noindent 
{\bf Conjecture.}
Suppose that $\alpha,q \in (-1,1)$. 
The probability measure $\nu_{\alpha,q}$ is 
freely infinitely divisible 
if and only if
if and only if $\alpha =q$ or $\alpha<q\geq0$. 

\medskip
This conjecture is guaranteed to be true in the restricted subfamilies $\{\nu_{\alpha,0} \mid \alpha\in(-1,1)\}$ (\cite[Theorem 3.2]{SY01}), $\{\nu_{0,q}\mid -1< q<1\}$ (\cite{ABBL10} and \cite[Example 3.11]{AH13} for the free infinite divisibility), and $\{\nu_{q,q} \mid q\in(-1,1)\}$ (all measures in this family are freely infinitely 
divisible since they are $q^2$-Gaussians). 

%
\section{Commutation relations among one-mode
$(\alpha,q)$-operators}\label{sec:aq-commutation}

\begin{definition}
Suppose that $\alpha, q\in (-1,1)$ 
and $f$ is analytic on $\mathbb C$. 
\begin{enumerate}[\rm(1)]
\item
Let $Z$ be the multiplication operator defined by 
\begin{equation*}
	(Zf)(z):=zf(z). 
\end{equation*}
\item
Let $D_q$ be the Jackson derivative given by
\begin{equation*}
	(D_q f)(z)=
\begin{cases}
	\displaystyle\frac{f(z)-f(q z)}{(1-q)z}, & z\ne 0,\\
	f'(0), & z=0.
\end{cases}
\end{equation*}
\item
The $\alpha$-deformed Jackson derivative is given as 
\begin{equation*}
D_{\alpha,q}:=
\begin{cases}
	 D_q + \alpha q^{2 N} D_{1/q}, & q\ne 0, \\
	 D_0 + \alpha\frac{d}{d z}\big|_{0}, & q=0,
\end{cases}
\end{equation*}
where $N$ is the number operator. 
For $q\neq0$, we can also write 
\begin{equation*}
D_{\alpha,q}= D_q + \frac{\alpha}{q^2} D_{1/q} q^{2 N}. 
\end{equation*}
\end{enumerate}
\end{definition}

\begin{remark}
It is easy to check that the 
$\alpha$-deformed Jackson derivative 
is equivalently defined as 
\begin{equation*}
	(D_{\alpha,q}f)(z) = (D_q f)(z)
	+\alpha(D_{1/q} f)(q^2 z),
	\quad q\ne 0.
\end{equation*}
For example, if $f(z)=z^n$,
$(D_{\alpha,q}f)(z)=(1+\alpha q^{n-1})[n]_qz^{n-1}$
holds.  In fact,
the $\alpha$-deformed Jackson derivative is an analogue of 
the operator in \cite[Theorem2.5]{BEH15}.
\end{remark}

Then, one can realize one-mode analogue of  
$(\alpha,q)$-operators on an appropriate domain
of the one-mode interacting Bargmann-Fock space $\mathcal B$ 
with $\omega_n=(1+\alpha q^{n-1})[n]_q$ and $\alpha_n=0$ by 
$$a^+:=Z, \ a^-:=D_{\alpha,q}, \ \text{and} \ \Phi_n:=\frac{z^n}{\sqrt{[\omega_n]!}}.$$
In fact, it is easy to check that 
\begin{equation*}
\begin{cases}
	a^+\Phi_n = \sqrt{\omega_{n+1}}\Phi_{n+1},\\
	a^{-}\Phi_n = \sqrt{\omega_n}\Phi_{n-1},
\end{cases}
\end{equation*}
hold and the $q$-commutation relation,
one-mode analogue of \eqref{eq:aq-commutation},
\begin{align*}
	[a^-,a^+]_q\Phi_n 
	& := (a^-a^+-qa^+a^-)\Phi_n
	\\
	& = (I+\alpha q^{2N})\Phi_n,
\end{align*}
is satisfied.
Let us put 
$M_{\alpha,q}= I+\alpha q^{2N}$
and then one can get the expression, 
\begin{equation*}
	\displaystyle M_{\alpha,q} 
	= (1+\alpha)I-\alpha (1-q^2)Z D_{q^2},
\end{equation*}
due to  $(ZD_{q^2})\Phi_{n}=[n]_{q^2}\Phi_{n}$.

Therefore one can obtain the following
\begin{theorem}
Suppose $\alpha\in (-1,1)$ and $q\in (-1,1)$.
Then the following are satisfied. 

\begin{enumerate}[\rm(1)]
\item
$[a^-,a^+]_q= M_{\alpha,q}, \ \
[a^-, M_{\alpha,q}]_{q^2}=(1-q^2)a^-,\ \ 
[M_{\alpha,q},a^+]_{q^2}=(1-q^2)a^+$.
\item
$\displaystyle M_{\alpha,q} 
= (1+\alpha)I
-\alpha (1-q^2) 
Z D_{q^2}$.
\item
In particular, if $\alpha=q$, then 
one can obtain a more refined relation,
$[a^-,a^+]_{q^2}= (1+q)I$.
\end{enumerate}
\end{theorem}

\begin{example}
\begin{enumerate}[\rm(1)]
\item
$\alpha=0$ implies $[a^-,a^+]_q=I$.
Hence $M_{0,q}=I$ commutes with both $a^+$ and $a^-$,
\begin{equation*}
	[a^-, M_{0,q}]_{1}=[M_{0,q},a^+]_{1}=0.
\end{equation*}

Therefore, the case $\alpha\ne 0$ provides 
non-trivial commutation relations.
\item
If $\alpha=-q^{2\beta}$ for $\beta> 0$,
then the limiting case of the scaled operator 
is obtained as
\begin{align*}
	\lim_{q \uparrow 1}
	\frac{M_{-q^{2\beta},q}}{1-q^2} 
	& = \lim_{q \uparrow 1}
	\frac{I-q^{2\beta}q^{2N}}{1-q^2}= N+\beta. 
\end{align*}
Moreover, let us consider the scaled operators, 
\begin{equation*}
	A^\pm:= \lim_{q\uparrow 1}
	\frac{a^{\pm}}{\sqrt{1-q^2}}.
\end{equation*}
Then one can get  
\begin{equation*}
[A^-,A^+]_1 = N+\beta
\end{equation*}
and hence  
\begin{equation*}
[A^-, N]_1=A^-, \ [N, A^+]_1=A^+. 
\end{equation*}
It should be noted that 
these are the commutation relations for the classical 
Meixner-Pollaczek polynomials with respect to 
the symmetric Meixner distribution in \eqref{eq:symmetric-meixner}.
See \cite{As08}. 
\end{enumerate}
\end{example}

\bigskip
\noindent
{\bf Acknowledgments.}
N. Asai was partially supported by Grant-in-Aid 
for Scientific Research (C) 23540131, 
JSPS.  \ M. Bo\.zejko was partially supported 
by the MAESTRO grant DEC-2011/02/A/
ST1/00119 and OPUS grant DEC-2012/05/B/ST1/00626 
of National Center of Science.
T. Hasebe was supported by Grant-in-Aid 
for Young Scientists (B) 15K17549, JSPS.
%
\appendix
\section{Appendix}\label{sec:appendix}
Let $\Sigma_n$ be the set of bijections $\sigma$ of 
the $2n$ points 
$\{\pm 1,\pm 2,\cdots,\pm n\}$ with $\sigma(-k)=-\sigma(k)$.
Equipped with the composition operation as a product,
$\Sigma_n$ becomes what is called 
a Coxeter group of type B.
It is generated by $\pi_0:=(1,-1)$
and $\pi_i:=(i,i+1), \ 1\leq i\leq n-1$, 
which 
satisfy the generalized braid relations
\begin{equation}\label{eq:coxeterB}
\begin{cases}
	\pi_i^2 =e, & 0\leq i\leq n-1,\\
	(\pi_0\pi_1)^4
	 =(\pi_i\pi_{i+1})^3=e, & 1\leq i\leq n-1,\\
	(\pi_i\pi_j)^2
	 =e, & |i-j|\geq 2, \ 0\leq i,j\leq n-1.
\end{cases}
\end{equation}
An element $\sigma\in \Sigma_n$ expresses an irreducible 
form,
\begin{equation*}
	\sigma=\pi_{i_1}\cdots\pi_{i_k},
	\quad 0\leq i_1,\dots,i_k\leq n-1,
\end{equation*}
and in this case
\begin{align*}
	\ell_1(\sigma)
	& := \text{the number of $\pi_0$ in $\sigma$},\\
	\ell_2(\sigma)
	& :=\text{the number of $\pi_i, \ 1\leq i\leq n-1$,
	in $\sigma$}
\end{align*}
are well defined. Let $H$ be a separable Hilbert space. For  a given self-adjoint involution $f \mapsto \overline{f}$ 
for $f\in H$, 
an action of $\Sigma_n$ on $H^{\otimes n}$ is defined by
\begin{equation}\label{eq:pi-action}
\begin{cases}
	\pi_0(f_1\otimes \cdots \otimes f_n)
	= \overline{f_1}\otimes f_2\otimes\cdots \otimes f_n, & n\geq 1,\\
	\pi_i (f_1\otimes \cdots \otimes f_n)
	= f_1\otimes \cdots  \otimes f_{i-1} 
	\otimes f_{i+1} \otimes f_i  \otimes f_{i+2}\otimes\cdots\otimes f_n,
	& n\geq 2, \ 1\leq i\leq n-1.
\end{cases}
\end{equation}
The $(\alpha,q)$-inner product on the full Fock space $\mathcal{F}(H)$
is defined by 
\begin{equation}\label{eq:aq-inner product}
	\langle f_1\otimes \cdots \otimes f_m, 
	g_1\otimes\cdots \otimes g_n\rangle_{\alpha,q}
	:=\delta_{m,n}\sum_{\sigma\in \Sigma_n}
	\alpha^{\ell_1(\sigma)}q^{\ell_2(\sigma)}
	\prod_{j=1}^n
	\langle 
	f_j, g_{\sigma(j)}
	\rangle_H ,
	\ \alpha, q\in (-1,1)
\end{equation}
with conventions $0^0=1$ 
and $g_{-k} = \overline{g_k}, \ k=1,2,\dots,n$.
For example,  
if one may define the involution as $\overline{f}:=-f$,
then  $g_{-k}=-g_k$. Equipped with this inner product the full Fock space $\mathcal{F}(H)$ is denoted by $\mathcal{F}_{\alpha,q}(H)$ to emphasize on the dependence of the inner product on $\alpha,q$. 

The $(\alpha,q)$-creation operator $B^+_{\alpha,q}(f)$ is the usual left creation operator on the full Fock space, and the $(\alpha,q)$-annihilation operator $B^-_{\alpha,q}(f)$ is its adjoint with respect to the inner product $\langle\cdot , \cdot \rangle_{\alpha,q}$. They satisfy the commutation relation
\begin{equation}\label{eq:aq-commutation}
	B^-_{\alpha,q}(f)B^+_{\alpha,q}(g)-qB^
	+_{\alpha,q}(g)B^-_{\alpha,q} (f)
	=\langle f,g\rangle_H I
	+\alpha\langle \overline{f},g\rangle_H q^{2N}, 
	\quad f,g\in H.
\end{equation}
The readers can consult \cite{BEH15} for details.

%

\end{document}